\documentclass[12pt]{amsart}

\usepackage{amssymb}
\usepackage{amsmath, amsthm}
\usepackage{thmtools}
\usepackage{thm-restate}
\usepackage{mathrsfs}

\theoremstyle{plain}
\declaretheorem[name=Theorem,numberwithin=section]{theorem}
\newtheorem{lemma}{Lemma}[section]
\declaretheorem[name=Proposition, numberwithin=section]{proposition}
\newtheorem{corollary}{Corollary}[section]

\theoremstyle{definition}
\newtheorem{definition}{Definition}[section]

\newtheorem{conjecture}{Conjecture}[section]

\theoremstyle{remark}
\newtheorem{remark}{Remark}[section]

\usepackage[utf8]{inputenc}

\usepackage[
    backend=biber,
    style=alphabetic,
    sorting=nyt,
    maxalphanames=3,
    minalphanames=1,
    hyperref=true
]{biblatex}

\usepackage{textcase}   

\DeclareFieldFormat{labelalpha}{%
  \iffieldundef{labelnamescount}
    {#1}
    {%
      \ifnum\thefield{labelnamescount}=1
        \expandafter\splitSingle #1\relax
      \else
        \MakeUppercase{#1}%
      \fi
    }%
}

\def\splitSingle#1#2#3#4#5\relax{%
  \MakeUppercase{#1}\MakeLowercase{#2#3}#4#5%
}

\usepackage[colorlinks=true,citecolor=blue]{hyperref}
\title{On Dynamical Mordell-Lang Conjecture for split map}
\author{Yuanming Zhang}

\begin{document}

\begin{titlepage}
\thispagestyle{empty}
\maketitle

\begin{abstract}
In this article,  we prove the dynamical Mordell-Lang conjecture for split endomorphism $F=f_1\times\cdots\times f_n$ on the product space $C_1\times\cdots\times C_n$ defined over $\mathbb{C}$, where each $C_i$ is a curve with $f_i\in \text{End}_\mathbb{C}(C_i)$. Moreover, we proved that there is some effectively computable upper bound on the common difference of the arithmetic progressions. 

\end{abstract}

\tableofcontents

\end{titlepage}



\section{Introduction}

The Mordell-Lang Conjecture, originally formulated in the context of Diophantine geometry, asserts that the intersection of a finitely generated subgroup of an abelian variety with a closed subvariety is finite unless the subvariety contains a translate of a positive-dimensional abelian subvariety. This profound result, proven by Faltings for abelian varieties and extended by Vojta to semiabelian varieties, has inspired numerous analogs in arithmetic dynamics.

In the dynamical setting, the Dynamical Mordell-Lang Conjecture (DML), which is proposed by Ghioca and Tucker in \cite{ghioca2009periodic}, replaces group actions with iterations of an endomorphism and the finitely generated group by the orbit of a single point: 
\begin{conjecture}
Let \(X\) be a quasi-projective variety over an algebraically closed field \(K\) of characteristic zero, \(\Phi: X \to X\) an endomorphism, \(V \subset X\) a closed subvariety, and \(\alpha \in X(K)\) a starting point. Then set \(\{n \in \mathbb{N} : \Phi^n(\alpha) \in V\}\) is a finite union of arithmetic progressions. 
\end{conjecture}
This dynamical analog bridges arithmetic geometry and complex dynamics and connects to recurrence sequences as Scanon-Mahler-Lech type question.

Significant progress has been made on DML since its inception. Validation of the conjecture for the case $X=\mathbb{A}^2\ , \ \Phi=f_1\times f_2 , V=\triangle$, where $f_1, f_2$ are polynomials and $\triangle=\{(x,x) \ | \ x\in \mathbb{A}^1 \}$ the diagonal is made in \cite{GTZ07}, for the case of étale maps on quasi-projective varieties \cite{BGT10}, where the absence of ramification allows p-adic interpolation techniques to bound orbit intersections. The technics was used in the case of semiabelian varieties under group endomorphisms in \cite{Ghioca13}, also in the case of extensions to skew-linear maps on \(\mathbb{P}^1\) in \cite{GX18}. The case of birational polynomial morphism on affine plane was resolved by \cite{X14}, and later the case polynomial endomorphisms on $\mathbb{A}^2$ \cite{X15} which generalized \cite{GTZ07}, leveraging theory of valuation tree and intersection theory. 

In positive characteristic, the case becomes more complicated. She Yang proved the conjecture for totally inseparable liftings of Frobenius in \cite{Yang2024}. Analogs accounting for endomorphisms on algebraic tori with curves had been made in \cite{GhiD19}, with the times of intersection a finite union of arithmetic progression along with finitely many $p$-arithmetic sequences of the form $\{a+bp^{kn}\ | \ n\in \mathbb{N}\}$ for some positive integer $k$. 

Comprehensive treatments of these results appears in the monograph by Bell, Ghioca, and Tucker, which also explores sparse sets in orbits and positive characteristic variants \cite{BGT16} , also in \cite{xie2023around}.

Particular attention has been given to split maps, where the dynamics decouple across coordinates. For instance, on \(\mathbb{A}^n\) or \((\mathbb{P}^1)^n\), a split self-map \(\Phi = (f_1, \dots, f_n)\) acts independently on each factor. Known cases include the case $f_1=f_2$ is a rational function without periodic critical points\cite{BGKT12}. Pakovic in the case $X=(\mathbb{P}^1)^2 
\ ,\ V=\triangle$ and $\Phi=f\times g$ where $f,g$ are in the set of tame rational functions, which is Zariski open in the space of rational function of given degree\cite{P23}. For products of endomorphisms on an affine curve and a projective curve over \(\overline{\mathbb{Q}}\) \cite{XYZ26}, the conjecture holds, building on height inequalities related to the supper-attracting behavior. These developments pave the way for resolving DML on \((\mathbb{P}^1)^n\) under split rational maps, where the product structure allows reduction to one-dimensional dynamics while addressing cross-coordinate dependencies through arithmetic and analytic tools.

Other results concern the case for split maps are dynamical Manin mumford conjecture for split maps on $(\mathbb{P}^1)^2$(over $\mathbb{C}$ ) and the dynamical Bogomolov conjecture(over $\overline{\mathbb{Q}}$)\cite{GNY19}, and more recently the uniform dynamical Bogomolov conjecture for 1-parameter families of endomorphism on $(\mathbb{P}^1)^2$\cite{MS25}.

In this article, we advance towards a complete resolution of the Dynamical Mordell-Lang Conjecture for split endomorphisms on \((\mathbb{P}^1)^n\) over fields of characteristic zero. Our method extend the treatment of DML in the case of 'almost PCF' in \cite{BGT16} : The main theorem is the following

\begin{theorem}
    \label{mainthmone}
    For a finite set of irreducible curves $\{ C_i \}_{i=1,\cdots ,n}$ with set of endomorphisms $\{ f_i \ | \ f_i:C_i\rightarrow C_i\}$ all defined over $\mathbb{C}$, the \text{DML} conjecture holds for $(C_1\times \cdots \times C_n \ , \ f_1\times \cdots \times f_n)$.
\end{theorem}  

The key proposition for proving the theorem is the following

\begin{restatable}{proposition}{prop1}
\label{prop1}
    
    For a finite set of rational functions $f_1,\cdots,f_n$ defined over a number field with the sets of critical points $c_{i,j}\in K \ , \ j=1\cdots ,t_i$ of $f_i$ for $i=1,\cdots , n$. Up to some iteration we can assume every periodic critical point of some $f_i$ is a fixed point.Then we have a $\text{Gal}(\overline{K}/K)$-invariant set $S_{i,j}$ for each critical point $c_{i,j}$ and a Galois element $\gamma\in \text{Gal}(\overline{K}/K)$ acts fixed-point freely on each $S_{i,j}$ s.t. $S_{i,j}=f_i^{-\mathcal{N}}(c_{i,j})$ if $c_{i,j}$ is not fixed by $f_i$; $S_{i,j}=f_i^{-(\mathcal{N}-1)}(f^{-1}_{i}(c_{i,j})\backslash \{ c_{i,j} \})$ otherwise, for some positive number $\mathcal{N}$.

\end{restatable}

While statistical  on orbits in higher dimensions provides heuristics for unlikely avoiding ramification via random models\cite{BGHKST13},
Our proposition predicts that a given point will eventually be mapped into the \'{e}tale locus under iteration of reduction map for a set of primes of positive density:

\begin{proposition}
    For any finite set of rational functions $f_1,\cdots,f_n$ of \mbox{} $\mathbb{P}^1$ defined over some number field $K$, any point $x=(x_1,\cdots, x_n)\in (\mathbb{P}^1)^n(\overline{K})$, there is some finite extension $L/K$, a set of primes $\Pi$ of positive density of $\mathcal{O}_L$ and some $m_0$ s.t. for any $\mathfrak{p}\in \Pi$, $F=f_1\times\cdots\times f_n$ has good reduction \text{mod} $\mathfrak{p}$ and the forward orbit  $\{ F_\mathfrak{p}^{\circ (m+m_0)}(\overline{x_1},\cdots, \overline{x_n})\ | \ m=0,1,\cdots , \}$ does not contains any critical points of $F_\mathfrak{p}$ , unless some $x_i$ lies in the pre-image of some fixed critical point of $f_i$.
\end{proposition}

We can deduce the dynamical Mordell-Lang conjecture for split endomorphism and the distribution of attracting periodic points for rational functions over number field from the proposition:

\begin{restatable}{theorem}{mainthm} 
    \label{mainthmone}
    For any endomorphism $F=f_1\times\cdots \times f_m$ on $X=(\mathbb{P}^1)^n$ defined over $\mathbb{C}$ , any subvariety $Z\subseteq X $. If there is a point $x=(x_1,\cdots,x_n) \in X(\mathbb{C})$ with $|O_f(x)\ \cap \ Z|=\infty$, then $T=\{ m \ |  \ f^m(x)\in  O_f(x)\ \cap \ Z \}$ is a finite union of arithmetic progression with a finite set of points.
\end{restatable}

\begin{restatable}{corollary}{cor1}
    \label{coro1}
    For a rational function $f$ defined over some number field $K$, there is some finite extension $L/K$ and a set of primes $P$ of $L$ of positive density, s.t. any periodic point of $f$ is either $\mathfrak{p}$-indifferent or supper-attracting for any prime $\mathfrak{p}\in P$.

\end{restatable}

This can be viewed as a generalization of the fact that the multipliers of a PCF functions are $S$-unit for a finite set of primes $S$.  

Further more, by explicit analysis we can give an effective version of the conjecture, which tell us the period of the subvariety can be bounded explicitly in terms of the information of the base field and Weil height of the endomorphism:

\begin{restatable}{corollary}{cor2}
    \label{cortwo}
    For rational functions $f_1,\cdots, f_n/K$ with set of critical points $c_{i,j}\in K \ , \ i=1,\cdots ,n \ , \ j=1,\cdots , t_i$ and $x=(x_1,\cdots,x_n)\in K^n$ for a number field $K$. The upper bound of the period of \\  $\cap_{m\ge 0}\overline{O_F(f^m(x))}$ under $F=f_1\times\cdots\times f_n$ is effective computable from $[K:\mathbb{Q}],d_i,h_{M_{d_i}}(f_i)$. In particular, for any subvariety $S\subseteq (\mathbb{P}^1)^n$, if $O_F(x)\cap S$ is an infinite set, then the starting number and the common difference of $\{ m\ | \ F^m(x)\in S , m\in \mathbb{N} \}$ are effective bounded above.

\end{restatable}

Our method also gives some partial result towards the 'zero density conjecture':

\begin{definition}
    For a rational function $f\in K(x)$ and any $x\in \mathbb{P}^1(K)$ for a number field $K$, define
    \begin{align*}
        G_n(f,x) & :=\text{Gal}(K(f^{-n}(x)/K) \ , \ G_\infty(f,x):= \text{Gal}(\cup_{n\in \mathbb{N}}K(f^{-n}(x))/K) \\
        &=\lim_{n\rightarrow \infty} G_n(f,x).
    \end{align*}
\end{definition}

\begin{conjecture}
    (zero density) If $x$ is not periodic under $f$, then we have the density
    \begin{equation*}
        \rho(f,x):=\lim_{n\rightarrow \infty} \frac{\#\{ \sigma\in G_n(f,x) : \sigma \ \text{fixes at least one point in}\ f^{-n}(x) \}}{\# G_n(f,x)}
    \end{equation*}
    equals to $0$.
\end{conjecture}
The limit exists because of the martingale property associated to the random process as in the proof of the below corollary of proposition \ref{prop1} : 

\begin{restatable}{corollary}{cor3}

    For a rational function $f/K$, $x\in K$ for a number field $K$, if $x$ is not periodic under $f$, then there is an effective computable number $0<M(K,f)<1$, s.t. 
    \begin{equation*}
        \rho(f,x)\le M(K,f)< 1.
    \end{equation*}

\end{restatable}

\begin{proof}
    By the martingale property of the random process $\{X_n(f,x)\}$ defines on the ${G_n(f,x)}$ :
    \begin{equation*}
        X_n(f,x)(\sigma)=\begin{cases}
            1& \ , \ \text{if}\ \sigma \ \text{has a fixed point on some n-th preimage}; \\
            0& \ , \ \text{otherwise}.
        \end{cases}
    \end{equation*}
    for the map $f:f^{-n}(x)\rightarrow f^{-(n-1)}(x)$, the sequence in limit that defines $\rho(f,x)$ is non-deceasing. Thus if one $X_n(f,x)$ is not a constant $0$ function, then the limit is a positive number. This is ensured by proposition \ref{prop1} and the effectiveness is followed from the computability of step $n$ where a fixed point free Galois action on the $n$-th pre-image appear.
    
\end{proof}

we refer the reader to \cite{J14} for further exposition to the density results associates to arboreal representation and stability of iteration.

{\raggedright\textbf{Acknowledgment}} The author gratefully acknowledge his doctoral advisor Guanyuan Zhang for his introduction to Dynamical Mordell-Lang conjecture and encouragement during the first three years of his doctoral study. Special thanks also due to Prof. Junyi Xie for his insightful comment and suggestion for the early draft of this paper, in particular for the effective version of the conjecture.

\newpage

\section{Local Absolute Galois Group Associated to pre-iamges of Critical Point}

We give a brief introduction to the knowledge of local Galois group and arithmetic dynamics and prove the key theorem in this section. The main reference for this section is \cite{Serre79},\cite{FV02} and \cite{Silverman07}.

\subsection{Residue class field}

For a $\mathfrak{p}$-adic local field $K=K_\mathfrak{p}$, ring of integers $\mathcal{O}_\mathfrak{p}$, unique maximal ideal $\mathfrak{p}$, uniformizer $\pi_K$, denote its residue class field $k_\mathfrak{p}=\mathcal{O}_\mathfrak{p}/\mathfrak{p}$, which has $q=p^f$ many residue classes. The absolute Galois group $\text{Gal}(\overline{k}_\mathfrak{p}/k_\mathfrak{p})=\hat{\mathbb{Z}}$ , is a profinite group topologically generated by the \textbf{Frobenius automotphism} $\phi: x \mapsto x^q$.

\subsection{Maximal unramified extension}

The maximal unramified extension of the $p$-adic local field is by adjoining all roots of unity with order coprime to $p$: $K^{\text{ur}}=K(\{\zeta_n:(n,p)=1\}) $, where $\zeta_n$ is a primative $n$-th roots of unity. The Galois group $\text{Gal}(K^{\text{ur}}/K)=\hat{\mathbb{Z}}=\prod_{l\ \text{prime}}\mathbb{Z}_l$, topologically generated by the Frobenius lift $\sigma$, acting on the roots of unity by $\sigma : \zeta_n \mapsto \zeta^{q}_n$ .
The $\mathbb{Z}_l$ is the ring of $l$-adic integer, which is the inverse limit $\underleftarrow{\lim}_n \mathbb{Z}/l^n\mathbb{Z}$ with canonical projection $p_{ij}: \mathbb{Z}/l^i\mathbb{Z}\longrightarrow \mathbb{Z}/l^j\mathbb{Z}\ , \ i\ge j$. And $\hat{\mathbb{Z}}=\underleftarrow{\lim}_n \mathbb{Z}/n\mathbb{Z}$ is the profinite completion of the ring of integers.

\begin{remark}
    \item The action of the Galois group on units is determined by the \textbf{cyclotomic characteristic} $\chi_{\ell,K}: G_K\longrightarrow \mathbb{Z}_{\ell}^{\times}$($\ell\ne p$) defined by
    \begin{equation*}
        g(\zeta_{\ell^n})=\zeta_{\ell^n}^{\chi_{\ell,K}(g) \ \text{mod} \ \ell^n},
    \end{equation*}
    with $\zeta_{\ell^n}$ is a primitive $\ell^n$-th root of unity. The topological generator $\sigma$(Frobenius lift) of $\hat{\mathbb{Z}}$ satisfies $\chi_{\ell,K}(\sigma)=q \ \text{mod}\ \ell^{\infty}$, thus the generator is the diagonal multiplication by $q$.
\end{remark}

\subsection{Maximal tamely ramified extension}

The maximal tamely ramified extension $K^{tr}/K$ is obtained by $K(\{\pi_K^{1/n},\zeta_n:(n,p)=1\})=K^{ur}(\{\pi_K^{1/n}:(n,p)=1\})$.

The tame inertial group $I_t=\text{Gal}(K^{tr}/K^{ur})$ is isomorphic to $\hat{\mathbb{Z}}^{(p')}=\prod_{q\ne p}\mathbb{Z}_{q}$. And the full(tamely ramified) absolute Galois group \\ 
$\text{Gal}(K^{tr}/K)$ fits into the short exact sequence:

\begin{equation*}
 0\rightarrow   I_t\rightarrow\text{Gal}(K^{tr}/K)\rightarrow\text{Gal}(K^{ur}/K)\rightarrow 0,
\end{equation*}
with $I_t=\text{Gal}(K^{tr}/K^{ur})\cong\hat{\mathbb{Z}}^{(p')}$, $\text{Gal}(K^{ur}/K)\cong \hat{\mathbb{Z}}$. The sequence splits and gives the semi-direct product 
\begin{equation*}
    \text{Gal}(K^{tr}/K)\cong \hat{\mathbb{Z}}^{(p')}\rtimes \hat{\mathbb{Z}},
\end{equation*}  
with two topological generator $\sigma$(Frobenius lift) and $\tau$, subject to the relation :
\begin{equation*}
    \sigma\tau\sigma^{-1}=\tau^q.
\end{equation*}

\begin{remark}
    \begin{enumerate}
        \item The inertial group $\text{Gal}(\overline{K}/K^{ur})=P_K\rtimes I_t$, where $P_K$ is the pro-$p$ wild inertial group(corresponds to wild ramification part) and $I_t$ is the tame inertial group. $P_K$ is a pro-$p$ Demu\v{s}kin group , could be finitely or infinitely generated;
        \item The basic property about the two generators: $\tau(\zeta_n)=\zeta_n ,  \\ \tau(\pi_K^{1/n})=\zeta_n\pi_K^{1/n}$,
        $\sigma(\zeta_n)=\zeta_n^q$, for some primitive $n$-th root of unity ($(n,p)=1$). This two generator descend to algebraic generators for finite extension.
    \end{enumerate}
\end{remark}

\subsection{Local p-adic dynamics}

We investigate the Galois property associated to the local dynamics of $p$-adic attracting fixed point germ :

\begin{definition}
   A power series $f(x)=a_0+a_1x+\cdots \in \mathcal{O}_K[[x]]$ defines a bounded holomorphic function on $\mathrm{m}_{\mathbb{C}_{\mathfrak{p}}}$ , we define the \textbf{Weierstrass degree} of $f$ to be 
   \begin{equation*}
       \text{wdeg}(f):=\min_{i\in \mathbb{N}}\{i\ | \ a_i\in \mathcal{O}_{K}^{\times} \} \ \text{or} \ +\infty
   \end{equation*}
   if such integer does not exist.
\end{definition}

For a rational function with good separable reduction at some prime $\mathfrak{p}$ of the base field, we have an integral model of the system over the local ring, and the restriction of rational function to analytic germ commutes with the reduction map, thus the Weierstrass degree of any Taylor series associated to the function is finite. The degree tells us the sum of multiplicity of critical point of the map on the disk $\mathrm{m}_{\mathbb{C}_\mathfrak{p}}$ .

\begin{proposition}

    For a power series $f\in \mathcal{O}_K[[x]]$ associated to a rational function $F$ with good reduction at $\mathfrak{p}$, with $\text{wdeg}(f)=d>1$ and the characteristic of the residue field of $K$ is greater than $d$. Then any $a\in \mathrm{m}_K$ with $f(a)\ne a$, there is some integer $n_0$ s.t. for any $n>n_0$ and any roots $\alpha$ of $f^n=a$, $\tau(\alpha)\ne \alpha$ for the inertial generator $\tau \in \text{Gal}(K^{tr}/K)$. 
\end{proposition}

\begin{proof}

by classification of $p$-adic Fatou component of rational function \cite{RiveraLetelier04}, there must be some(unique) attracting fixed point $x_0$ , and the Taylor series expansion at $x_0$ is $x_0+\sum_{i=1}^{\infty} b_i(x-x_0)^i$, and $v_\mathfrak{p}(b_1)>0$ if $b_1\ne 0$. Thus $|f(t)-x_0|_\mathfrak{p}\le \max\{|b_i(t-x_0)^i|_\mathfrak{p} \ | \ i=1,2,\cdots \ \}\le |t-x_0|_\mathfrak{p}^2 \ $ or $|b_1(t-x_0)|_\mathfrak{p}$ (possible when $b_1 \ne 0$). 

Fix a valuation on $\mathbb{C}_\mathfrak{p}$ s.t. $v_\mathfrak{p}(\pi_K)=1$. Any positive sequence(we start back-ward iteration at positive distance from $x_0$) and bounded by $1$ and satisfies $a_{n+1}\ge (a_n)^{1/2} \ \text{or} \ a_{n}/|b_1|_\mathfrak{p}$ is monotone increasing to $1$\ ($|b_1|_\mathfrak{p}<1$ ) , thus there is some $n_0$,
s.t. any root $\alpha$ of $f^{n_0+m}(x)=a$ has $\mathfrak{p}$-adic valuation less than $1$. By analyzing the Newton polygon of the power series at each time of taking pre-image under $f$, the denominator of the negative slopes are all not divisible by the characteristic, thus $K(\alpha)/\alpha$ is tamely ramified and $\tau(\alpha)\ne \alpha$.

\end{proof}

\subsection{Local Galois property of sets of critical points }

Notation as above, assume we have $f_1,\cdots , f_n$ rational functions defined over some number field $K$(contains all the critical points of each $f_i$). Denote $c_{i,j}\ , \ i=1,\cdots ,n \ , \ j=1,\cdots , t_i $ (distinct)sets of critical points of each $f_i$, with ramification degree $d_{i,j}$ of $f_i$ at $c_{i,j}$. 

In the sequel we choose a prime $\mathfrak{p}$ of $K$ s.t. all $f_i$ has good reduction at $\mathfrak{p}$, and the characteristic of $k=\mathcal{O}_{K_\mathfrak{p}}/\mathfrak{p}$ is greater than $\prod d_{i.j}$. We may assume every preperiodic critical point is either fixed or (strictly)preperiodic after taking a common times of iteration. By '\textbf{periodic}' we mean $f^n(x)=x$ for some non-negative integer $n$; '\textbf{strictly preperiodic}' we mean it's preperiodic but not periodic.

\begin{lemma} 
For a rational function $f$ defined over some number field $K$ and the prime $\mathfrak{p}$ of $K$ chosen as above, we have 

  \begin{enumerate}
      \item 
    \begin{equation*}
        \overline{f^{-1}(x)}=\overline{f}^{-1}(\overline{x})\ , \ \overline{f_{*}(x)}=\overline{f}_*(\overline{x})  
    \end{equation*}
    as divisors for any $x\in \mathbb{P}^{1}(\overline{K})$ ;
    \item $\overline{f^{-n}(D)}=\overline{f}^{-n}(\overline{D}) \ , \ \overline{\text{{Crit}}}(f^n)=\text{Crit}(\overline{f^n})$ for any divisor $D$ and positive integer $n$. ($\text{{Crit}}(f):=$ the support of the different divisor of $f$ )
\end{enumerate}
\end{lemma}

\begin{proof}
By the choice of the prime $\mathfrak{p}$ of $K$ s.t. $f$ has good reduction and $\overline{f}$ is tamely ramified. Then $f$ extends to a finite flat morphism over the local ring $\mathcal{O}_{K_\mathfrak{p}}$, with generic fiber $f$ and special fiber $\overline{f}$. The first conclusion follows immediately from taking tensor product with residue field.

Then the third equality holds since the n-th iteration of the integral model is a model for $f^n$.

For the last equality, it's essentially reduced to a polynomial equation that defines the critical points and the corresponding non-zero congruence equation mod the prime $\mathfrak{p}$, whose solution are all liftable.

\end{proof}

 The above lemma ensures that for all but finitely many primes of $K$, the congruence equation $\overline{f^N}(x)=\overline{c}$ defines the same divisor as the reduction of the set of solution $f^N(x)=c$ , so the estimation on the upper bound is explicit.

At this end we prove proposition 1.1 which is the key technical result:
The ideal of the proof is to treat the back-ward orbit of a critical points in 2 parts : one part is the points that are strictly preperiodic mod $p$ , the back-ward orbit of such point is mutually distinct mod $p$; second part is the 'ramification part', that corresponds to periodic critical points under reduction, and we use the $p$-adic analysis and inertial group of $p$-adic Galois group to tackle this part.

\begin{proof}[Proof of proposition \ref{prop1} ]
    For each $f_i$ and critical point $c_{i,j}$ of $f_i$: if $\overline{c_{i,j}}$ is periodic under $\overline{f_i}$, denote its (minimal)period by $p_{i,j}$, write all the critical points $c_{i,j_1},\cdots,c_{i,j_{l_{i,j}}}$ of $f_i$ s.t. there is some $0\le k_v<p_{i,j} $ with $\overline{f^{k_v}_i(c_{i,j})}=\overline{c_{i,j_v}}$, for $1\le v \le l_{i,j}$ ($j_1=j$). Then for the $\mathfrak{p}$-adic unit disk $D_{c_{i,j}}$(with base field $\overline{K_\mathfrak{p}}$) corresponds to $\overline{c_{i,j}}$, $g_{i,j}:=f_{i}^{\circ p_{i,j}}|_{D_{c_{i,j}}}$ defines a holomorphic germ in $\mathcal{O}_{K_\mathfrak{p}}[[x-c_{i,j}]]$ with $\text{wdeg}(g_{i,j})=\prod_{v=1}^{l_{i,j}} d_{i,j_v}$. By lemma 2.1, for every positive integer $n$, every point in $\overline{f_{i}^{-np_{i,j}}(c_{i,j})\backslash g_{i,j}^{-n}(c_{i,j})}$ is a strictly preperiodic under $\overline{f_i}$.
    
    Denote all the non-fixed critical points $c_{i,j}$ s.t. $\overline{c_{i,j}}$ is periodic under $\overline{f_{i}}$ by $\{c_{i,j_{k}}\ | \ i=1,\cdots, n ,1\le k\le l_i \}$($l_i=-1$ if $f_i$ has no such critical point). The period $p_{i,j_k}$ and holomorphic germ $g_{i,j_k}$ with $\text{wdeg}(g_{i,j_k})=w_{i,j_k}$ as above. Then by proposition 2.2 , there are some integer $n_{i,j_k}$ s.t. for any $m> n_{i,j_k}$ and any roots $\alpha$ of $g_{i,j_k}^{\circ m}=c_{i,j_k}$, $\tau(\alpha)\ne \alpha$.
    Let $p:=\prod_{i,j_k}p_{i,j_k}$. So sufficiently large $m$, the inertial element $\tau $ satisfies $\tau(\alpha)\ne \alpha$ for every number $\alpha$ in $g^{\circ(-mp/p_{i,j_k})}_{i,j_k}(c_{i,j_k})$ . Let $f$ be a positive integer s.t. $\sigma^f(\alpha)=\alpha$ for any such root $\alpha$ and
    let $\widetilde{\gamma}=\tau \sigma^f$, then $\widetilde{\gamma}(\alpha)=\tau(\alpha)\ne \alpha$.

     Fix such $m$, for any point $x$ in 
     $f_{i}^{-mp}(c_{i,j_k})\backslash g_{i,j_k}^{-mp/p_{i,j_k}}(c_{i,j_k})$ or strictly periodic $\overline{c_{i,j}}$ is strictly preperiodic under $\overline{f_i}$, the backward orbit sequence $\{ x_t\}$ with $x_1=x$, $f_i(x_{t+1})=x_t$ , the points in $\{ \overline{x_t} \}$ are distinct. Since any finite extension of $k_\mathfrak{p}$ is a finite field, for any positive integer $h$, there is a integer $N$, s.t. for any $m'>N$, any such back-ward orbit sequence $\{ \overline{x_t} \}$ associated to some $c_{i,j}$ in the above mentioned case satisfies $[k_\mathfrak{p}(x_m')/k_\mathfrak{p}]>h$. If we take $h>\Tilde{q}^f$($\Tilde{q}=\# \mathcal{O}_\mathfrak{p}/\mathfrak{p}$ ), then any $x_m$ satisfies $\sigma_\mathfrak{p}^f(x_m)=x_m^{\Tilde{q}^f}\ne x_m$.
     
     Set 
     
      \[
     S_{i,j}=
     \begin{cases}
         f_i^{-2Nmp}(f_i^{-1}(c_{i,j}) \backslash \{ c_{i,j} \} ) & \text{if} \ f_i(c_{i,j})=c_{i,j}; \\
         f_i^{-(2Nmp+1)}(c_{i,j}) & \text{otherwise} .
     \end{cases}
     \]
     
     Firstly $S_{i,j}$ is $\text{Gal}(\overline{K}/K)$-invariant ($\text{Gal}(\overline{K}/K)(S_{i,j})=S_{i,j}$) for any $i,j$ , since $c_{i,j}\in K$. Let $\gamma=\tilde{\gamma}|_{K(\{ S_{i,j}\})}$. For any $x\in S_{i,j}$: consider the sequence $\overline{x},\overline{f_i(x)},\cdots , \overline{f_i^{2Nmp}(x)}$, if there is some $\overline{f_i^{t_1}(x)}=\overline{f_i^{t_2}(x)} \ , \ t_1<t_2$, then there is some $0\le t_0<2Nmp$, s.t. $\overline{x},\overline{f_i(x)},\cdots , \overline{f_i^{t_0-1}(x)}$ are distinct, $\overline{f_i^{t_0}(x)}=\overline{f_i^{t_x}(c_{i,j})}$ for some $0\le t_x< p_{i,j}$.
     
     If $t_0>N$, since there is a canonical projection $\mathscr{p}: \text{Gal}(\overline{K_\mathfrak{p}}/K_\mathfrak{p})\rightarrow \text{Gal}(\overline{k_\mathfrak{p}}/k_\mathfrak{p})$ which commutes with the projection $\mathbf{p}: \overline{K_\mathfrak{p}}\rightarrow \overline{k_\mathfrak{p}}$, $\mathbf{p}(g(x))=\mathscr{p}(g)(\overline{x})=\overline{x}^{q^f}\ne \overline{x}$, thus $x\ne \gamma(x)$.

     If $2Nmp-t_0>(m+1)p$, then there exists some $t_0\le v < 2Nmp$, s.t. $f_i^{v}(x)\in g_{i,j}^{\circ (-mp/p_{i,j}-1)}(g_{i,j}^{-1}(c_{i,j})\backslash \{ c_{i,j} \})$, then $f_i^v(\gamma(x))=\gamma(f_i^v(x))=\tau\sigma^f(\alpha)=\tau(\alpha)\ne \alpha =f_i^v(x)$, thus $g(x)\ne x$.
     Since $N+(m+1)p<2Nmp+1$, one of the cases must happen and we conclude that $\gamma$ acts on each $S_{i,j}$ fixed point freely. 
    
\end{proof}

\section{Behavior of forward orbit and distribution of attracting periodic points}

\begin{theorem}
    (\text{Chebotarev Density theorem}) Let $L/K$ be a Galois extension of number fields, denote $G:=\text{Gal}(L/K)$. For a conjugating invariant subset $C\subseteq G$, define
    \begin{equation*}
        \Pi_C(x,L/K):= \# \{ \mathfrak{p} \ : \ N(\mathfrak{p})\le x , \mathfrak{p} \ \text{is unramified in} \ L/K , \ \text{and}\ \sigma_\mathfrak{p}\subseteq C \}, 
    \end{equation*}
    where $N(\mathfrak{p})$ is the numerical norm of prime ideal $\mathfrak{p}$ of $K$, $\sigma_\mathfrak{p}$ is the Frobenius conjugacy class corresponding to $\mathfrak{p}$ in $\text{Gal}(L/K)$. Then the natural density 
    \begin{equation*}
        \lim_{x\rightarrow \infty} \frac{\Pi_{C}(x,L/K)}{\Pi_G(x,L/K)}=\frac{|C|}{|G|}.
    \end{equation*}

\end{theorem}

\begin{lemma}
    For a Galois-invariant set of algebraic numbers $S=\{ x_i\} \subset \overline{K}$, if there is some 
 $\sigma\in \text{Gal}(K(S)/K)$ acts fixed point freely on $S$. Then there is some prime $\mathfrak{p}$(actually infinitely many) of $K$ s.t. the reduction $\overline{S}=\{ \overline{x_1},\cdots , \overline{x_n}\}$ satisfies $\overline{S}\ \cap \ k_\mathfrak{p}=\emptyset$, where $k_\mathfrak{p}=\mathcal{O}_K/\mathfrak{p}$.
\end{lemma}

\begin{proof}
 Denote the conjugacy class of $\sigma$ in $\text{Gal}(K(S)/K)$ by $\text{Cl}(\sigma)$ , $\prod_{i\ne j}(x_i-x_j)=D$. Denote the set of prime ideals of $K$ that is below at least one prime factor $D$ in $K(x_1,\cdots, x_n)$ by $P$.
By Chebotarev density theorem , there is some $x>0$ s.t. there is some prime $\mathfrak{p}\in [\Pi_{\text{Cl}(\sigma)}(K(S)/K)\backslash P]$. So there is some prime $\mathfrak{P}$ above $\mathfrak{p}$, with corresponding Frobenius elements  $\sigma$ s.t. $\overline{\sigma(x_i)}=(\overline{x_i})^{q}$, where $q=\# \mathcal{O}_\mathfrak{p}/\mathfrak{p}$. 

If $\overline{x_{i_0}}\in k_\mathfrak{p}$ for some $x_{i_0}$, then $\overline{x_{i_0}}=(\overline{x_{i_0}})^q=\overline{\sigma(x_{i_0})}$, which implies that $\mathfrak{P} \ | \ \sigma(x_{i_0})-x_{i_0}$, this contradicts to $\mathfrak{P}\not\in P$. Thus $\overline{x_i}\not\in k_\mathfrak{p}$ for any $x_i \in S$. 

\end{proof}

As a consequence of the theorem, given any endomorphism \\ $(f_i)_{i=1,\cdots,n}$ of $(\mathbb{P}^1)^n$ defined over some number field, any closed point will eventually goes into a quasi-periodic domain of the endomorphism under iteration, for a positive proportion of primes(after some finite field extension):

\begin{proposition}
    For any finite set of rational functions $f_1,\cdots,f_n$ of \mbox{} $\mathbb{P}^1$ defined over some number field $K$, any point $x=(x_1,\cdots, x_n)\in (\mathbb{P}^1)^n(\overline{K})$, there is some finite extension $L/K$, a set of primes $\Pi$ of positive density of $\mathcal{O}_L$ and some $m_0$ s.t. for any $\mathfrak{p}\in \Pi$, $F=f_1\times\cdots\times f_n$ has good reduction \text{mod} $\mathfrak{p}$ and the forward orbit  $\{ F_\mathfrak{p}^{\circ (m+m_0)}(\overline{x_1},\cdots, \overline{x_n})\ | \ m=0,1,\cdots , \}$ does not contains any critical points of $F_\mathfrak{p}$ , unless some $x_i$ lies in the pre-image of some fixed critical point of $f_i$.
\end{proposition}

\begin{proof}

We extend $K$ by $K(x_1,\cdots,x_n)$. By proposition \ref{prop1}, there exist a positive number $N$ and the corresponding $\text{Gal}(\overline{L}/L)$-invariant sets $S_{i,j}$ and a Galois element $g\in \text{Gal}(\overline{K}/K)$ acts fixed point freely on each $S_{i,j} $.
Thus $g$ also acts fixed point freely on $S=\cup S_{i,j}$. 
By lemma 3.1 there exists a set of prime $P$ of positive density for some finite field extension $L/K$, s.t. for any $\mathfrak{p}\in P$ , $f_{i,j}$ have good reduction and $\overline{S}\ \cap \ l_\mathfrak{p}=\emptyset$, where $l_\mathfrak{p}=\mathcal{O}_L/\mathfrak{p}$. Then $(t_1,\cdots,t_n)$ is a critical point of $F_\mathfrak{p}$ iff $t_i=\overline{c_{i,j}}$ for some $i,j$.  

We assume that there is no $x_i$ lies in the back-ward orbit of any fixed critical point of $f_i$, then there is some number $N$, s.t. $F^{N+m}(x)\not \in \text{Crit}(F)$ for any $m>0$. Replace $x$ by $F^N(x)$. 

Exclude finitely many primes in $P$ s.t. $F^v_\mathfrak{p}(\overline{x})\in \text{Crit}(F_\mathfrak{p})$ for some $0\le v\le 2Nmp+1$, and obtain $P'$, with the term $m,p,N$ the same as in theorem 2.2. If there is some $v$ and s.t. $F_\mathfrak{p}^v(\overline{x})\in \text{Crit}(F_\mathfrak{p})$, then $v>2Nmp+1$ and there is some $i,j$ with 
\newline
$\overline{f_i^v(x_i)}=\overline{f_i}^{\circ (2Nmp+1)}(\overline{f_i^{v-2Nmp-1}(x_i)})=\overline{c_{i,j}}$. 

For all but finitely many primes in $P'$ we have $\overline{f_i^{v-2Nmp-1}(x_i)}\in \overline{f_i}^{\ \circ -(2Nmp+1)}(\overline{c_{i,j}})=\overline{f_i^{\ \circ -(2Nmp+1)}(c_{i,j})}$.By lemma 2.1, if $f_i(c_{i,j})\ne c_{i,j}$ : then
\newline
$\overline{f_i^{\ \circ -(2Nmp+1)}(c_{i,j})}=\overline{S_{i,j}}$ and $\overline{f_i^{v-2Nmp-1}(x_i)}\in \overline{S_{i,j}}\ \cap \ l_\mathfrak{p} $, a contradiction; if $c_{i,j}$ is a fixed point : we can find some $w>0$ , s.t. $\overline{f_i^{2Nmp+w}(x_i)}\ne \overline{c_{i,j}}$ and $\overline{f_i^{2Nmp+1+w}(x_i)}=\overline{c_{i,j}}$,
we can deduce that $\overline{f_i^{2Nmp+w}(x_i)}\in \overline{f^{-1}_{i,j}(c_{i,j})\backslash \{ c_{i,j}\}}$, similarly we can draw a contradiction. The proposition is proved.

\end{proof}

Together we can prove the corollary \ref{coro1} which gives information on the  attracting property of periodic points of a rational function over number field:

\begin{proof}[Proof of corollary \ref{coro1}]
    After replacing $f$ by some iteration $f^n$, we assume that all the periodic critical point is fixed. Let $f_i=f \ , \ i=1,\cdots , E \ , \ F=f^{\times E}$, $x=(c_1,c_2,\cdots , c_{E})$, where $\text{Crit}(f)=\{ c_1,\cdots , c_{E} \}$ (distinct critical points) and $c_1,\cdots, c_E$ are all the non-preperiodic critical points of $f$. By proposition 3.1, there is a set of primes of $L$ of positive density, and a positive integer $n$, s.t. for any $\mathfrak{p}\in P$ and any $m>n$, $\overline{F}^m(\overline{x})\not \in \text{Crit}(\overline{F})=\overline{\text{Crit}(F)}$, this implies that $f^{m}(c_i)\not \equiv c_i \ (\text{mod} \ \mathfrak{p})$ for all $m>n$ if $c_i$ is not fixed, which means there is no $\mathfrak{p}$-adic attracting periodic point: for a periodic point $w$ with $f^m(w)-w=0 , \ f^m(w)'\not =0$, if $\overline{(f^m(w))'}= 0$, we have $\overline{w}\in \text{Crit}(\overline{f^m})=\overline{\text{Crit}(f^m)}$, so there must be some critical point of $f^m$ in the $\mathfrak{p}$-cycles of unit disks, but all critical point of $f^m$ mapped to some critical point of $f$ under iteration, thus the critical point is periodic under reduction, so it must be periodic under $f$, contradicts to the attracting property.
\end{proof}

The following result is due to \cite{Poo14}, which generalized theorem 3.3 in \cite{BGT10} :

\begin{proposition}
    Let $K$ be a field complete w.r.t an absolute value $| \cdot |$ with residue field of characteristic $p$ and $|p|=1/p$. Denote $R$ the valuation ring of $K$ and $R \langle\mathbf{x}\rangle$ the \textbf{Tate algebra} of $R$ in $n$-variables, with $\mathbf{x}=(x_1,\cdots,x_n)$, which consists of power series $f=\sum_{\mathbf{i}\in \mathbb{N}^n}f_\mathbf{i}\mathbf{x}^\mathbf{i}$ converging on the closed unit polydisk, i.e. $|f_\mathbf{i}|\rightarrow 0$. 
    
    Then if $f\in R\langle\mathbf{x}\rangle^n$ satisfies $f(\mathbf{x})\equiv \mathbf{x} \ (\text{mod} \ p^c)$ for some $c> 1/(p-1)$, then there exists $g\in R\langle\mathbf{x},t\rangle^n$ satisfies $g(\mathbf{x},t)=f^t(\mathbf{x})\in R\langle\mathbf{x}\rangle^n$ for any $t\in \mathbb{N}$.
    
\end{proposition}

The proposition tells us that we can have an interpolation along the orbit of any point where the reduced morphism is \'etale and the reduced point is fixed. Thus there is some ananlytic curves that passes through all the points in the orbit. As a simple corollary, we have the following:

For a scheme $S$, a subscheme $Z\subseteq S$, we call a infinite sequence of points $\{ x_i \}\subseteq S$ \textbf{generic} in $Z$, if any infinite subsequence of $\{ x_i \}\subseteq S$ is Zariski dense in $Z$. By the definition , $x_i\in Z$ for all but finitely many $i$.

\begin{proposition}
    For a finite extension $K/\mathbb{Q}_p$ , $\pi : X\rightarrow \text{Spec}(\mathcal{O}_K)$ a smooth scheme of finite type over ring of integer $\mathcal{O}_K$ of $K$ and $f: X\rightarrow X$ a morphism over $\text{Spec}(\mathcal{O}_K)$. Suppose $x\in X(R)$ and $f$ is \'{e}tale along the orbit of $x$, i.e. $\{ i \ | \ i\in \mathbb{N} \ , \ \overline{f}^i(\overline{x}) \in \text{Ram}(\overline{f}) \}$ is a finite set , $\overline{x} \ , \ \overline{f}$ is the specialization of $x$ and $f$ at the unique maximal ideal $\mathrm{m}_\mathfrak{p}$ of $\mathcal{O}_K$. Then there exists some subvarieties $Z_1 ,\cdots , Z_m$ with the same dimension s.t. $f^i(Z_k)\subseteq Z_{k+i}$ (the index is taken \text{mod} $m$ ) , and $O_{f^m}(f^i(x))$ is generic in $Z_i$ for $i=1,2,\cdots , m$ . 
\end{proposition}

As a corollary of proposition 3.1 and 3.2, we can prove the theorem \ref{mainthmone} :

\begin{proof}[Proof of theorem \ref{mainthmone}:]
    If $x$ is a periodic point of $f$ then the result is clear. We assume $O_f(x)$ is infinite in the following : 
    
    If the $f,Z,x$ are defined over some number field(up to a Mobi\"{u}s transformation) : we can assume that $x_1,\cdots ,x_u$ are the coordinates that are preperiodic under the corresponding $f_i$ of $f$, with the preperiodic relation $f_j^{s_j}(x_j)=f_j^{s_j+p_j}(x_j)$ for $j=1,\cdots, u$. By proposition 3.1 and 3.2:  there are positive dimensional subvarieties $W'_1\ , \cdots , W_m'\subseteq \mathbb{P}^{(n-u)}_{\overline{\mathbb{Q}}}$, s.t. $(f')^{\circ k}(W'_i)\subseteq W_{i+k}' $ and for $x'=(x_{u+1},\cdots ,x_{n})\ , \ f'=f_{u+1}\times\cdots \times f_n$ , $O_{(f')^{\circ m}}((f')^{\circ i}(x'))$ is generic in $W'_i$ for $i=1,\cdots ,m$. Denote $[m,p_1,\cdots , p_u]$ by $M$, then the $O_{f^{\circ M}}(f^t(x))$ is generic in the subvariety $W_t:=\{ (f^{\circ M}_1(f_1^t(x_1)),\cdots ,$ 
    $f_u^{\circ M}(f_u^t(x_u)))\ | \ k\in \mathbb{N}\}\times W'_t$ for $t=1,\cdots , [m,p_1,\cdots , p_u]$ .($[ \cdot ,\cdot  ]$ is the symbol for least common multiple) Since $Z\ \cap \  O_{f}(x)$ is an infinite set, there must be some $t\in \{ 1,\cdots , M \}$ , s.t. $Z \cap O_{f^M}(f^t(x))$ is also infinite, then $W_t\subseteq Z$ and $f^{kM}(f^t(x))\in Z$ for all but finitely many $k\in \mathbb{N}$. Then $T=\{ kM+t \ | \ k\in \mathbb{N}\ , \ W_t\subseteq Z  \}\backslash S$ for some finite set $S$. The conclusion holds in this case.
    In general the field generated by all the coefficients of each $f_i$ and $x_i$ is a finitely generated field $F/\mathbb{Q}$, which can be viewed as a rational function field of some smooth quasi-projective variety $\mathfrak{B}/K$ for some number field $K$. Then each $x_i$ can be viewed as a regular section of the structure morphism $\pi : \mathbb{P}_{\overline{\mathbb{Q}}}^1\times \mathfrak{B} \rightarrow \mathfrak{B}$ with $p_2\circ x_i=id_{\mathfrak{B}}$ , also the $f_i$ can be viewed as a rational section of the map $\mathbf{p}_i : \text{Rat}_{d_i}/\overline{\mathbb{Q}}\ \times \mathfrak{B}\rightarrow \mathfrak{B}$ induced by the coefficients of $f_i$ , with $p_2\circ \mathbf{p}_i=id_\mathfrak{B}$ and $\text{deg}(f_i)=d_i$. We can assume that all the section $\mathbf{p}_i$ are regular section after excluding the locus where some $f_i$ is degenerated. In other words, the $\mathfrak{B}$ parametrized a family of dynamical system. 
    
    Similarly assume that $x_1,\cdots,x_u $ are all the preperiodic marked point then there is a parameter $c\in \mathfrak{B}(\overline{K})$, s.t. $x_{j}(c)$ is not preperiodic under $f_j(c)$ for $j=u+1,\cdots ,n$: first choose a parameter $c_0\in \mathfrak{B}(\overline{K})$ where all the section are smooth assume $x_{u+1}(c_0),\cdots , x_{v}(c_0)$ are all the preperiodic points among $x_{u+1}(c_0),\cdots , x_n(c_0)$. Choose a prime $\mathfrak{p}$ of $K(c_0)$ (residue field of $\mathfrak{B}$ at $c_0$) s.t all $f_i(c_0)$ have good reduction at $\mathfrak{p}$. First each $x_t(c_0)$ will be mapped to some invariant cycles of $\mathfrak{p}$-adic unit disks under $f_t(c_0)$ for $t=v+1,\cdots , n$, and there are finitely many periodic points of $f_t(c_0)$ in the cycles (multiple periodic points corresponds to 'parabolic' periodic point for reduced map $\overline{f_t(c_0)}$); each periodic point in the cycle is the specialization of some periodic marked point $c_{t,l}$ of $f_t$ at $c_0$.  Consider a small disk $D$ (in the sense of $\mathfrak{p}$-adic analytification of $\mathfrak{B}$) at the parameter $c_0$: the equation $x_t=c_{t,l}$ defines a codimension at least $1$ analytic subvarieties in $D$ for each $t$, so it suffices to choose $c$ in $D$ outside the union of such subvaireties.
    
    By proposition 3.1, there is a set of primes $P$ of $K(c_0)$ of positive density and positive integer $m_0$, s.t. $F(c_0)=f_1(c_0)\times \cdots \times f_n(c_0)$ has good reduction at any prime $\mathfrak{p}$ in $P$, and $O_{\overline{f_j}}(\overline{f^{m_0}(x_{j}(c_0))})\ \cap \ \text{Crit}(\overline{f_j})=\emptyset$. For such prime $\mathfrak{p}$: denote $K(c_0)$ by $L$, first do a base change $\mathfrak{B}\otimes L_\mathfrak{p}=\mathfrak{B}_{L_\mathfrak{p}}$, then there is a unit disk $D_{c_0}\subset \mathfrak{B}_{L_\mathfrak{p}}^{an}$ corresponds to the residue class $\overline{c_0}\in \mathfrak{B}_{L_\mathfrak{p}}\otimes l_\mathfrak{p}$, where $\mathfrak{B}^{an}$ denotes the analytification w.r.t the metric on $L_\mathfrak{p}$. If we can choose a closed point in $c'\in D_{c_0}(L_\mathfrak{p})$ that maps to the generic point of $\mathfrak{B}$ via the projection $\mathfrak{B}_{L_\mathfrak{p}}\rightarrow \mathfrak{B} $. The value of the coefficients and the marked points at $c'$ are all not equals to $\infty$, and the value of each function defines an embedding $K(\mathfrak{B})\rightarrow L_\mathfrak{p}$. We do the base change of $(f,x)$ to from $K(\mathfrak{B})$ to $L_\mathfrak{p}$ via the embedding, then we obtain a system $\pi: (\mathbb{P}^1)^n\rightarrow \text{Spec}(\mathcal{O}_L)$ of smooth scheme of finite type and $F$ a morphism over $\text{Spec}(\mathcal{O}_L)$, with $(f_{\mathfrak{m}_\mathfrak{p}},x_{\mathfrak{m}_\mathfrak{p}})=((\overline{f(c_0)},\overline{x(c_0)})$, where $\mathfrak{m}_\mathfrak{p} \subset L_\mathfrak{p}$ is the open unit disk. The rest is similar to the number field case . 
    
    The existence of such closed point $c'$ follows from the fact that the base field is complete non-trivially metrized , thus the space $\mathfrak{B}(L_\mathfrak{p})$ is a Baire space, then it suffices to choose closed point that lies outside a countable union of proper closed subvarieties, which is meager. 
    
    The theorem is proved. 
\end{proof}

As a result, we prove that the DML holds for product of endomorphisms on any dimension $1$ curves over $\mathbb{C}$:

\begin{proof}[Proof of theorem \ref{mainthmone}]
    We may assume all $f_i$ are dominant, otherwise we can reduce the number of coordinates as in the previous periodic coordinate case. By taking normalization and projective closure we obtain smooth and projective model $\overline{C_i}$, then extend each $f_i$ to endomorphism $\overline{f_i}:\overline{C_i}\rightarrow \overline{C_i}$.
    If there is some $\overline{C_i}$ with $g(\overline{C_i})>1$, then some iteration of $\overline{f_i}$ equals to identity, by passing same iteration of the product map we again can reduce the number of coordinates. 

    Assume genus $g(\overline{C_i})=1$ for $i=1,\cdots ,m$, $g(\overline{C_j})=0$ for $j=m+1,\cdots , n$ and the point $x=(x_1,\cdots,x_n)\in \prod \overline{C_i}$ has no preperiodic $x_i$ under $\overline{f_i}$, and the dynamical system $(\overline{C_1}\times\cdots \overline{C_m}\ ,\ F_m=\overline{f_1}\times\cdots\times \overline{f_m})$ is defined over some finitely generated field $L_1/L$, similarly $F_n=\overline{f_{m+1}}\times\cdots\times \overline{f_{n}}$ defined over some $L_2/L$ for some number field $L$. Then by proof of theorem 3.2, there exists an infinite set $P$ of primes of some number field $K$, and some field embedding $i_\mathfrak{p}: L_2\hookrightarrow K_\mathfrak{p}$ for each $\mathfrak{p}\in P$, s.t. after do the base change of $F_n$ via the embedding, we obtain a smooth dynamical system of finite type over $\text{Spec}(\mathcal{O}_{K_\mathfrak{p}})$, and the orbit of $(x_{m+1},\cdots,x_n)$ under reduction over $\mathfrak{m}_\mathfrak{p}$ does not goes into the critical locus when the times of iteration larger than some given number $m_0$. Extend $i_\mathfrak{p}$ to $L_1L_2\hookrightarrow K_\mathfrak{p}$ s.t. 
    the base change of $F_m$ also has good reduction over $\mathfrak{m}_\mathfrak{p}$. Since $F_m$ is \'{e}tale, the system $(F_m\times F_n)/K_\mathfrak{p}$ has good reduction over $\mathfrak{m}_\mathfrak{p}$ and the orbit of $\overline{F^{m_0}(x)}$ does not go into the critical locus under reduction map. We conclude proof by proposition 3.1.

\end{proof}

\section{Effectiveness of the upper bound on the period}

 For a Galois extension $L/K$ of number fields, any unramified prime ideal $\mathfrak{p}$ of $K$, denote $[\frac{L/K}{\mathfrak{p}}]=\text{Frob}_\mathfrak{p}$ the \textbf{Artin symbol} of $\mathfrak{p}$ which is the a conjugacy class of $G=\text{Gal}(L/K)$. For each conjugacy class $C$ of $G$, define

 \begin{enumerate}
     \item 
     \begin{equation*}
     \pi_C(x,L/K)=\left | \left \{  \mathfrak{p}:\mathfrak{p}\ \text{unramified in} \ L \ , \left [\frac{L/K}{\mathfrak{p}} \right ]=C , \ \text{N}_{K/\mathbb{Q}}\mathfrak{p}\le x  \right \} \right | \ ;
     \end{equation*}
     \item (Logarithmic integral)
     \begin{equation*}
         \text{Li}(x)=\int_2^x\frac{\text{d} t}{\log t} \sim \frac{x}{\log x} \ , \ x\rightarrow \infty .
     \end{equation*}

 \end{enumerate}
 
 The following is due to [LO77]:

 \begin{theorem}
     For a Galois extension $L/K$ of number fields, if $x\ge \exp{(10 n_L(\log d_L)^2)}$, then 
     \begin{equation*}
         \left |  \pi_C(x,L/K)-\frac{|C|}{|G|}\text{Li}(x) \right | \le \frac{|C|}{|G|}\text{Li}(x^\beta)+c_1 x\exp{(-c_2n_L^{-\frac{1}{2}}(\log x)^{\frac{1}{2}})},
     \end{equation*}
     with 
     \begin{equation*}
        \beta< \max \left [1-(16\log d_L)^{-1},1-(c_3d_L^{1/n_L})^{-1} \right  ],
     \end{equation*}
    where $n_L=[L:\mathbb{Q}]$ and $d_L$ is the discriminant of $L/\mathbb{Q}$. The constants $c_1,c_2 , c_3$ are explicitly computable and depend on $L,K$,
 \end{theorem}

 As a corollary , we have the effective short interval Chebotarev density

 \begin{corollary}
     There are some effective computable constant $M=M(K,L)$, $C_1,C_2>0$, s.t. for any $x>M , y>0$ we have
     \begin{align*}
        \left |  \pi_C(x+y,L/K)- \pi_C(x,L/K) \right | \ge C_1\frac{x+y}{\log(x+y)}-C_2\left [\frac{x}{\log(x)}+\frac{(x+y)^\beta}{\beta\log(x+y)}\right ].
     \end{align*}

     In particular, there is some effective computable constant $T$, s.t. for any $y>Tx\ , \ x>M$ we have $\left |  \pi_C(x+y,L/K)- \pi_C(x,L/K) \right |>1$.
 \end{corollary}

 \begin{proof}
     We have
\begin{align*}
&\Biggl| \pi_C(x+y,L/K) - \frac{|C|}{|G|}\operatorname{Li}(x+y) 
      + \frac{|C|}{|G|}\operatorname{Li}(x+y) 
      - \frac{|C|}{|G|}\operatorname{Li}(x) 
      + \frac{|C|}{|G|}\operatorname{Li}(x) \\
&\qquad - \pi_C(x,L/K) \Biggr| \\
&\ge \frac{|C|}{|G|}\bigl[\operatorname{Li}(x+y)-\operatorname{Li}(x)\bigr] 
   - \Biggl| \pi_C(x+y,L/K) - \frac{|C|}{|G|}\operatorname{Li}(x+y) \Biggr| \\
&\qquad - \Biggl| \pi_C(x,L/K) - \frac{|C|}{|G|}\operatorname{Li}(x) \Biggr| \\
&\ge \frac{|C|}{|G|}\bigl[\operatorname{Li}(x+y)-\operatorname{Li}(x)\bigr] 
   - \frac{|C|}{|G|}\bigl[\operatorname{Li}((x+y)^\beta)+\operatorname{Li}(x^\beta)\bigr] \\
&\qquad - c_1\Bigl[\exp\bigl(-c_2 n_L^{-1/2}(\log x)^{1/2}\bigr) 
                + \exp\bigl(-c_2 n_L^{-1/2}(\log(x+y))^{1/2}\bigr)\Bigr]=(*),
\end{align*}

and 
\begin{equation*}
   \left (\frac{x}{\log x}\right)'=\frac{\log x-1}{(\log x)^2} \le\frac{1}{\log x}\le \frac{2(\log x-1)}{(\log x)^2}
\end{equation*}
when $x>10$, thus
\begin{equation*}
    \frac{x}{2\log x}\le\text{Li}(x)\le \frac{2x}{\log x}\ , x>100.
\end{equation*}
Also $(\log x)^{1/2}\ge \log(\log x)$ when $x>100$. So we have
\begin{align*}
    (*)\ge &\frac{|C|}{|G|}\left[\frac{x+y}{2\log(x+y)}-\frac{2x}{\log x}-\frac{2x^\beta}{\log x^\beta}-\frac{(x+y)^\beta}{2\log(x+y)^\beta}\right ]\\
    &-c_1\exp{(-c_2n_L^{-1/2})}[\log x+\log(x+y)],
\end{align*}
thus the result follows. 

 \end{proof}

We give an effective version of DML from above :

\begin{proof}[Proof of corollary \ref{cortwo}]

We may assume all the coordinate $x_i$ is not periodic.  For any $f\in \text{M}_d(\overline{\mathbb{Q}})$ defined over $K$, assume all critical points are defined over $K$ and denote $h_{\text{M}_d}(f)$ the module height of $f$, $h_{\text{Crit}}(f)$ the critical height of $f$, by results of \cite{ingram18} we have 
$h_{\text{M}_d}\asymp h_{\text{Crit}}$. Based on $|h(f(x))-d\cdot h(x)|\le C(h_{\text{M}_d}(f))$ and $|h(x)-h_f(x)|\le C'(h_{\text{M}_d}(f))$, where $d=\text{deg}(f)$, $h$ the Weil height on $\mathbb{P}^1$ and $h_f$ the dynamical canonical height of $f$, we have 
\begin{equation*}
    h(f^n(c))\le d^n h_f(c)+C'\le d^nh_{\text{Crit}}(f)+C'\le d^nc_3h_{\text{M}_d}(f)+d^nc_4+C'(h_{\text{M}_d}(f))
\end{equation*}

for any critical point $c$ of $f$, where $C''$ is a constant depends on $f,n$ . So for any prime ideal $\mathfrak{p}$ of $K$, we have 
\begin{align*}
v_\mathfrak{p}(f^n(c)-c)&\le [K:\mathbb{Q}]h(f^n(c)-c)\le [K:\mathbb{Q}][h(f^n(c))+h(c)+2]\\ &\le[K:\mathbb{Q}][(d^n+1)c_3h_{\text{M}_d}(f)+(d^n+1)c_4+2C'(h_{\text{M}_d}(f))+2].
\end{align*}

For any fixed point $x_0$ of $f^n$, $[K(x_0):K]\le d^n$ thus for any prime $\mathfrak{q}$ of splitting field of $f^n=x$ lies above $\mathfrak{p}$ of $K$, the ramification degree $e(\mathfrak{q}/\mathfrak{p})\le d^n$ and $v_\mathfrak{p}((f^n)'(x_0))\ge 1/d^n$(the extended valuation on $\overline{K_\mathfrak{p}}$).

If $x_0$ is a $\mathfrak{p}$-attracting fixed point: as in the argument of proposition 2.2, we have $v_\mathfrak{p}(c-x_0)<v_\mathfrak{p}(f^n(c)-x_0)$ , thus $v_\mathfrak{p}(c-f^n(c))=v_\mathfrak{p}(c-x_0)$. Thus the time $m$ s.t. any root $\alpha$ of $f^{nm}(x)=c$ in the $\mathfrak{p}$-adic disk $D_C$ has valuation less than any number in $\mathfrak{p}$ can be compute effectively from the $n,d,[K:\mathbb{Q}]$ and $h_{M_d}(f)$. And the degree $d$ is bounded by the product of all the degree at critical points. Also $n$ is bounded by $\# k_\mathfrak{p}$, which is bounded by $(\text{char} \ k_\mathfrak{p})^{[K:\mathbb{Q}]}$.

In order to achieve tamely ramification and good separable reduction at $\mathfrak{p}$ , it suffices to choose $\mathfrak{p}$ not divides $\text{Res}(f_i),\ i=1,\cdots , m $ and $\text{char} (k_\mathfrak{p})\nmid \prod (\sum_{j_1<j_2<\cdots < j_s}d_{i,j_k})\ , 1\le s\le t_i$. Choose $n_0$ s.t. any prime $\mathfrak{q}$ of $K$ with $N_{K/\mathbb{Q}}\ \mathfrak{q}>n_0$ satisfies the above property(note that $n_0$ can be bounded above from $h_{M_{d_i}}(f_i)$). The $n_0$ is effective computable since the resultant of $f_i$ is a regular function on the space of rational function of degree $d_i$ . Denote the classes of $\text{Gal}(L/K)$ obtained from proposition 3.1 for the rational functions $f_1\cdots ,f_n$ by $C$, where $L=(K(\{f_i^{-N}(c_{i,j}))\ | \ i=1,\cdots ,m ,\ j=1,\cdots , t_i\})$. By corollary 4.1, there is some effective computable constant $T(N,d_i,K)$, s.t. there is a prime \\
$\mathfrak{p}\in \pi_C(Tx_0,L/K)\backslash \pi_C(x_0,L/K)$. From the above argument, $N$ can be bounded above explicitly by $h_{M_{d_i}}(f_i),d_i,K$. Thus there is a prime $\mathfrak{p}_0$ of $K$ s.t. the orbit of $x$ under $F$ goes into some \'{e}tale cycles under reduction mod $\mathfrak{p}_0$, the results follows from proposition 3.3 .
    
\end{proof}

\printbibliography

@article{BGT10,
 ISSN = {00029327, 10806377},
 URL = {http://www.jstor.org/stable/40931051},
 author = {J. P. Bell and D. Ghioca and T. J. Tucker},
 journal = {American Journal of Mathematics},
 number = {6},
 pages = {1655--1675},
 publisher = {Johns Hopkins University Press},
 title = {THE DYNAMICAL MORDELL-LANG PROBLEM FOR ÉTALE MAPS},
 urldate = {2026-03-01},
 volume = {132},
 year = {2010}
}

@book{BGT16,
  author    = {Bell, Jason P. and Ghioca, Dragos and Tucker, Thomas J.},
  title     = {The Dynamical Mordell--Lang Conjecture},
  series    = {Mathematical Surveys and Monographs},
  volume    = {210},
  publisher = {American Mathematical Society},
  address   = {Providence, RI},
  year      = {2016},
  isbn      = {978-1-4704-2408-4},
  doi       = {10.1090/surv/210},
  note      = {136--142}
}

@book{FV02,
  author    = {Fesenko, Ivan B. and Vostokov, Sergei V.},
  title     = {Local Fields and Their Extensions},
  edition   = {2},
  series    = {Translations of Mathematical Monographs},
  volume    = {121},
  publisher = {American Mathematical Society},
  address   = {Providence, RI},
  year      = {2002},
  pages     = {xxii + 345},
  isbn      = {978-0-8218-3259-2},
  doi       = {10.1090/mmono/121}
}

@book{Serre79,
  author    = {Serre, Jean-Pierre},
  title     = {Local Fields},
  translator = {Greenberg, Marvin Jay},
  series    = {Graduate Texts in Mathematics},
  volume    = {67},
  publisher = {Springer-Verlag},
  address   = {New York},
  year      = {1979},
  isbn      = {0-387-90424-7},
  doi       = {10.1007/978-1-4757-5673-9},
  mrnumber  = {554237}
}

@book{Silverman07,
  author    = {Silverman, Joseph H.},
  title     = {The Arithmetic of Dynamical Systems},
  publisher = {Springer New York},
  year      = {2007},
  edition   = {1},
  series    = {Graduate Texts in Mathematics},
  volume    = {241},
  doi       = {10.1007/978-0-387-69904-2},
  isbn      = {978-0-387-69903-5}
}

@article{RiveraLetelier04,
  author    = {Rivera-Letelier, Juan},
  title     = {Sur la structure des ensembles de Fatou p-adiques},
  year      = {2004},
  eprint    = {math/0412180},
  archivePrefix = {arXiv},
  primaryClass = {math.DS},
  note      = {Preprint submitted on 8 Dec 2004}
}

@article{GTZ07,
  author   = {Dragos Ghioca and Thomas J. Tucker and Michael E. Zieve},
  title    = {Intersections of polynomial orbits, and a dynamical {Mordell--Lang} conjecture},
  journal  = {Inventiones mathematicae},
  volume   = {171},
  number   = {2},
  pages    = {463--483},
  year     = {2008},
  month    = feb,
  issn     = {1432-1297},
  doi      = {10.1007/s00222-007-0087-5},
  url      = {https://doi.org/10.1007/s00222-007-0087-5}
}

@article{J14,
     author = {Jones, Rafe},
     title = {Galois representations from pre-image trees: an arboreal survey},
     journal = {Publications math\'ematiques de Besan\c{c}on. Alg\`ebre et th\'eorie des nombres},
     pages = {107--136},
     year = {2013},
     publisher = {Presses universitaires de Franche-Comt\'e},
     doi = {10.5802/pmb.a-154},
     language = {en},
     url = {https://www.numdam.org/articles/10.5802/pmb.a-154/}
}

@article{X14,
  author  = {Xie, Junyi},
  title   = {Dynamical Mordell--Lang conjecture for birational polynomial morphisms on ${\mathbb {A}}^2$},
  journal = {Mathematische Annalen},
  volume  = {360},
  number  = {1},
  pages   = {457--480},
  year    = {2014},
  doi     = {10.1007/s00208-014-1039-1}
}

@article{X15,
  title={The Dynamical Mordell-Lang Conjecture for polynomial endomorphisms of the aﬃne plane},
  author={Junyi Xie},
  journal={Ast{\'e}risque},
  year={2015},
  url={https://api.semanticscholar.org/CorpusID:117971659}
}

@article{XYZ26,
  title         = {Dynamical Mordell--Lang conjecture for split self-maps of affine curve times projective curve},
  author        = {Xie, Junyi and Yang, She and Zheng, Aoyang},
  year          = {2026},
  month         = {02},
  eprint        = {2602.08608},
  archiveprefix = {arXiv},
  primaryclass  = {math.DS},
  doi           = {10.48550/arXiv.2602.08608}
}

@article{GX18,
  title     = {Algebraic dynamics of skew-linear self-maps},
  author    = {Ghioca, Dragos and Xie, Junyi},
  journal   = {Proceedings of the American Mathematical Society},
  volume    = {146},
  number    = {10},
  pages     = {4369--4387},
  year      = {2018},
  doi       = {10.1090/proc/14104}
}

@article{GNY19,
  title     = {The Dynamical Manin-Mumford Conjecture and the Dynamical Bogomolov Conjecture for split rational maps},
  author    = {Ghioca, Dragos and Nguyen, Khoa D. and Ye, Hexi},
  journal   = {Journal of the European Mathematical Society},
  volume    = {21},
  number    = {5},
  pages     = {1571--1594},
  year      = {2019},
  doi       = {10.4171/JEMS/869}
}

@article{BGHKST13,
  title     = {Periods of rational maps modulo primes},
  author    = {Benedetto, Robert L. and Ghioca, Dragos and Hutz, Benjamin and Kurlberg, P{\"a}r and Scanlon, Thomas and Tucker, Thomas J.},
  journal   = {Mathematische Annalen},
  volume    = {355},
  number    = {2},
  pages     = {637--660},
  year      = {2013},
  doi       = {10.1007/s00208-012-0799-8}
}

@article{BGKT12,
  title     = {A case of the dynamical Mordell-Lang conjecture},
  author    = {Benedetto, Robert L. and Ghioca, Dragos and Kurlberg, P{\"a}r and Tucker, Thomas J.},
  journal   = {Mathematische Annalen},
  volume    = {352},
  number    = {1},
  pages     = {1--26},
  year      = {2012},
  doi       = {10.1007/s00208-010-0621-4}
}

@article{MS25,
  title     = {On the dynamical Bogomolov conjecture for families of split rational maps},
  author    = {Mavraki, Niki Myrto and Schmidt, Harry},
  journal   = {Duke Mathematical Journal},
  volume    = {174},
  number    = {5},
  pages     = {803--856},
  year      = {2025},
  doi       = {10.1215/00127094-2024-0041}
}

@article{P23,
  title     = {Tame rational functions: Decompositions of iterates and orbit intersections},
  author    = {Pakovich, Fedor},
  journal   = {Journal of the European Mathematical Society},
  volume    = {25},
  number    = {10},
  pages     = {3953--3978},
  year      = {2023},
  doi       = {10.4171/JEMS/???}
}

@article{ingram18,
  title        = {The critical height is a moduli height},
  author       = {Ingram, Patrick},
  journal      = {Duke Mathematical Journal},
  volume       = {167},
  number       = {7},
  pages        = {1311--1346},
  year         = {2018},
  month        = may,
  doi          = {10.1215/00127094-2017-0053},
  url          = {https://doi.org/10.1215/00127094-2017-0053}
}

@article{Poo14,
  title = {p-adic interpolation of iterates},
  author = {Poonen, Bjorn},
  journal = {Bulletin of the London Mathematical Society},
  volume = {46},
  number = {3},
  pages = {525--527},
  year = {2014},
  publisher = {Wiley Online Library},
  doi = {10.1112/blms/bdu010},
  url = {https://doi.org/10.1112/blms/bdu010}
}

@article{ghioca2009periodic,
  author  = {Dragos Ghioca and Thomas J. Tucker},
  title   = {Periodic points, linearizing maps, and the dynamical {M}ordell--{L}ang problem},
  journal = {Journal of Number Theory},
  volume  = {129},
  number  = {6},
  pages   = {1392--1403},
  year    = {2009},
  doi     = {10.1016/j.jnt.2008.09.014},
}

@misc{xie2023around,
  title         = {Around the dynamical {M}ordell--{L}ang conjecture},
  author        = {Junyi Xie},
  year          = {2023},
  eprint        = {2307.05885},
  archivePrefix = {arXiv},
  primaryClass  = {math.AG},
  url           = {https://arxiv.org/abs/2307.05885}
}

@article{Ghioca13,
  title={The Mordell-Lang Question for Endomorphisms of Semiabelian Varieties},
  author={Dragos Ghioca and Thomas J. Tucker and Michael E. Zieve},
  journal={arXiv: Number Theory},
  year={2013},
  url={https://api.semanticscholar.org/CorpusID:1179384}
}

@article{GhiD19,
 author = {Ghioca, Dragos},
 title = {The dynamical {Mordell}-{Lang} conjecture in positive characteristic},
 fjournal = {Transactions of the American Mathematical Society},
 journal = {Trans. Am. Math. Soc.},
 issn = {0002-9947},
 volume = {371},
 number = {2},
 pages = {1151--1167},
 year = {2019},
 language = {English},
 doi = {10.1090/tran/7261},
 zbMATH = {6993263},
 Zbl = {1461.11092}
}

@article{Yang2024,
  author = {Yang, She},
  title = {Dynamical Mordell–Lang conjecture for totally inseparable liftings of Frobenius},
  journal = {Mathematische Annalen},
  year = {2024},
  volume = {389},
  number = {2},
  pages = {1639--1656},
  month = {6},
  note = {Published online: 01 June 2024},
  issn = {1432-1807},
  doi = {10.1007/s00208-023-02682-y},
  url = {https://doi.org/10.1007/s00208-023-02682-y}
}

\end{document}